\documentclass[12pt]{amsart}
\usepackage{fullpage,verbatim,amssymb}
\usepackage{hyperref, bbm}
\usepackage{soul}
\usepackage[active]{srcltx}
\usepackage[usenames,dvipsnames]{color}
\hypersetup{colorlinks=true,linkcolor=blue,citecolor=magenta}
\usepackage{ltablex}
\usepackage{longtable}
\usepackage{booktabs}

\makeatletter
\let\@@pmod\mod
\DeclareRobustCommand{\mod}{\@ifstar\@pmods\@@pmod}
\def\@pmods#1{\mkern4mu({\operator@font mod}\mkern 6mu#1)}
\makeatother

\definecolor{blue}{rgb}{0,0,1}
\definecolor{red}{rgb}{1,0,0}
\definecolor{green}{rgb}{0,.6,.2}
\definecolor{purple}{rgb}{1,0,1}

\long\def\red#1\endred{\textcolor{red}{#1}}
\long\def\blue#1\endblue{\textcolor{blue}{#1}}
\long\def\purple#1\endpurple{\textcolor{purple}{ #1}}
\long\def\green#1\endgreen{\textcolor{green}{#1}}

\newcommand{\ph}{\varphi}

\newcommand{\scrG}{\mathcal{G}}
\newcommand{\scrL}{\mathcal{L}}

\newcommand{\scrF}{\mathcal{F}}

\newcommand{\N}{\mathbb{N}}

\newcommand{\R}{\mathbb{R}}
\newcommand{\C}{\mathbb{C}}

\newcommand{\HH}{\mathbb{H}}

\renewcommand{\arg}{\textrm{Arg}}
\DeclareMathOperator{\SL}{SL}

\newcommand{\re}{{\text{Re}}}
\newcommand{\im}{{\text{Im}}}

\newcommand{\sm}{\left(\begin{smallmatrix}}
\newcommand{\esm}{\end{smallmatrix}\right)}
\newcommand{\bpm}{\begin{pmatrix}}
\newcommand{\ebpm}{\end{pmatrix}}

\newtheorem{theorem}{Theorem}
\newtheorem{lemma}[theorem]{Lemma}
\newtheorem{proposition}[theorem]{Proposition}
\newtheorem{corollary}[theorem]{Corollary}
\newtheorem{definition}[theorem]{Definition}

\theoremstyle{remark}

\newtheorem{remark}[theorem]{Remark}

\newtheorem*{example}{Example}

\numberwithin{theorem}{section}
\numberwithin{equation}{section}

\title{Derivatives of $L$-series of weakly holomorphic cusp forms}
\author{Nikolaos Diamantis} 
\address{University of Nottingham}
\email{nikolaos.diamantis@nottingham.ac.uk}
\author{Fredrik Str\"{o}mberg}
\address{University of Nottingham}
\email{fredrik.stromberg@nottingham.ac.uk}
\begin{document}

\maketitle

\begin{abstract}
Based on the theory of $L$-series associated with weakly holomorphic
modular forms in \cite{DLRR}, we derive explicit formulas for central
values of derivatives of $L$-series as integrals with limits inside
the upper half-plane. This has computational advantages, already in
the case of classical holomorphic cusp forms and, in the last section,
we discuss computational aspects and explicit examples.
\end{abstract}

\section{Introduction}
As evidenced by the prominence of conjectures such as those of Birch-Swinnerton-Dyer, Beilinson etc., central values of derivatives of $L$-series are key invariants of modular forms. Explicit forms of their values are therefore desirable, since they can lead to either theoretical or numerical insight about their nature.

On the other hand, an extension of classical modular forms that allowed for poles at the cusps, the \emph{weakly holomorphic modular forms}, has, more recently, been the focus of intense research, with Borcherd's work \cite{Borcherds} representing an important highlight followed by further applications to arithmetic, combinatorial and other aspects, e.g. in \cite{BruinierOno, Zw, BO, DIT} etc.

Up until relatively recently, $L$-series of weakly holomorphic modular forms had not been studied systematically. In fact, to our knowledge, a first definition was given in \cite{BFK} in 2014. In work by the first author and his collaborators \cite{DLRR}, a systematic approach for all harmonic Maass forms was proposed which led to functional equations, converse theorems etc.

A first application to special values of the $L$-series defined in \cite{DLRR} was given in \cite{DR}, where results of \cite{BFI} on cycle integrals were streamlined and generalised. Part of the work in \cite{BFI} was based on an explicit formula of what could be thought of as the (at
the time of writing of \cite{BFI}, not yet defined) central L-value of a weight $0$ weakly holomorphic
form. That formula had been suggested, in the case of the Hauptmodul, by Zagier. In \cite{DR} we interpreted those cycle integrals as values of the $L$-series defined in \cite{DLRR} and this allowed us to generalise the formulas of \cite{BFI}.

Here, we extend that study to values of \emph{derivatives} of  $L$-series of weakly holomorphic forms. To state the main theorem, we will briefly introduce the terms involved, but we will discuss them in more detail in the next section. 

Let $k \in 2\N$. We consider the action $|_k$ of $\SL_2(\R)$ on smooth functions $f\colon \HH \to \C$ on the complex upper half-plane $\HH$, given by 
\begin{equation*} 
(f|_k\gamma)(z):= j(\gamma, z)^{-k} f(\gamma z), \qquad \text{for $\gamma=\bpm a & b \\  c & d \ebpm \in$ SL$_2(\R)$}.
\end{equation*} 
where $j(\gamma, z):=cz+d$
. We further recall the defining formula for the Laplace transform $\scrL$
of a piecewise smooth complex-valued function $\varphi$ on $\R$. It is given by
\begin{equation}\label{e:Laplace_trans}
(\scrL \varphi)(s):=\int_0^{\infty} e^{-s t} \varphi(t)dt    
\end{equation}
for each $s \in \C$ for which the integral converges absolutely. We use the same notation $\scrL \varphi$ for its analytic continuation to a larger domain, if such a continuation exists. Finally, if $N \in \N$, 
$$W_N:=\sm 0 & -1/\sqrt{N} \\\sqrt{N} & 0 \esm. $$

Let now $f$ be a weakly holomorphic cusp form of weight $k$ for $\Gamma_0(N)$ with Fourier expansion 
 \begin{equation}\label{FourExInt}
f(z) = \sum_{\substack{n \ge -n_0}} a_f(n) e^{2\pi inz}.
\end{equation} 
Then the $L$-series of $f$ is defined in \cite{DLRR} as the map
$\Lambda_f$ given by
\begin{equation}\label{LseriesInt}
\Lambda_f(\ph)=\sum_{n \ge -n_0} a_f(n) (\scrL \ph)(2 \pi n) 
\end{equation}
for each $\varphi$ in a certain family of functions on $\R$  which will be defined in the next section.

The main object of concern in this note will be the specialisation of this $L$-series to a specific family of test functions: For $(s, w) \in \C \times \HH$  we denote
\begin{equation}\label{varphInt}\varphi_s^w(t):=\mathbf 1_{[1/\sqrt{N}, \infty)}(t) N^{s/2}e^{-wt}t^{s-1}, \qquad \text{for $t>0$}.
\end{equation}
where $\mathbf 1_X$ denotes the characteristic function of $X \subset \mathbb R.$ We then set
\begin{equation}\label{L(fInt)}
\Lambda(f, s):=\Lambda_f(\varphi_s^0)
\end{equation}
With this notation, we have
\begin{theorem}\label{mainwf} Let $k \in 2\N$ and $m \in \N$. For each weakly holomorphic cusp form of weight $k$ for $\Gamma_0(N)$ such that $f|_kW_N=f$ 
we have
\begin{equation*}
\Lambda^{(m)}(f,k/2)=i^{2m-\frac{k}{2}}N^{\frac{k}{4}}\sum_{j=0}^m  \binom{m}{j}\log^j \left (\frac{i}{\sqrt{N}}\right ) 
\int_{\frac{i}{\sqrt{N}}}^{\frac{i}{\sqrt{N}}+1}f(z) \zeta^{(m-j)} \left (1-\frac{k}{2}, z \right ) dz
\end{equation*}
where $\zeta(s, z)$ stands for the classical Hurwitz zeta function and $\zeta^{(r)}(s,z)=\frac{\partial^{r}}{\partial s^{r}} \zeta(s,z)$.
\end{theorem}

Our approach yields new expressions for derivatives of $L$-series of classical cusp forms too. 
Specifically, classical $L$-series can be  expressed in terms of the $L$-series associated with weakly holomorphic forms in \cite{DLRR} in the following way: For a classical cusp form $f$ of weight $k$ and level $N$ with $L$-series $L_f(s)$, we consider its completed $L$-function
$$L^*_f(s):=\left (\frac{\sqrt{N}}{2 \pi}\right )^s \Gamma(s) L_f(s).$$ 
Then, as verified in Sect. \ref{class}, we have
$$L^*_f(s)=\lim_{x \to 0^+}L_f(\ph_s^{ix}-\ph_{k-s}^{ix})$$
for $\ph_s^{ix}$ as in \eqref{varphInt}. Because of this, we can apply the method that led to Th. \ref{mainwf}, to deduce Th. \ref{main}, a special case of which is the following:
\begin{theorem}\label{cormain} For each weight $2$ cusp form $f$ of level $N$, such that $f|_2W_N=f$
we have
\begin{equation*}
(L_f^*)'(1)=2\sqrt{N} i\int_{\frac{i}{\sqrt{N}}}^{\frac{i}{\sqrt{N}}+1} f(z)\left ( \log(\Gamma(z))+(\log (\sqrt{N})-\pi i/2)z \right ) dz.
\end{equation*}
\end{theorem}
In particular, this formula interprets the central value of the first derivative as an integral with limits inside the upper half-plane. 
After providing the theoretical background in Section \ref{Prelim} and provide proofs of Theorems \ref{mainwf} and \ref{cormain} in Sections 
 \ref{sec:derivatives} and \ref{class} we will present some remarks regarding computational aspects, potential applications and numerical examples 
 of Theorems \ref{mainwf} and \ref{cormain} in the final section.

\section*{Acknowledgements}
We thank Dorian Goldfeld for helpful and encouraging comments on the manuscript. 
Part of the work was done while the first author was visiting Max Planck Institute for Mathematics in Bonn, whose hospitality he acknowledges. Research on this work is partially supported by the authors' EPSRC grants (ND: EP/S032460/1 FS: EP/V026321/1).  

\section{$L$-series evaluated at test functions}\label{Prelim}
In \cite{DLRR}, a new type of $L$-series was associated with general harmonic Maass forms and some basic theorems about it were proved. In this section we will provide relevant results in the special case which we need here, namely weight $k$ weakly holomorphic cusp forms for $\Gamma_0(N)$.  We require some additional definitions to describe the set-up.

Let $C(\R, \C)$ be the space of piecewise smooth complex-valued functions on $\R$.
For each function $f$ given by an absolutely convergent series of the form
 \begin{equation}\label{FourEx}
f(z) = \sum_{\substack{n \ge -n_0 \\ n \neq 0}} a_f(n) e^{2\pi inz},
\end{equation} 
we let $\scrG_f$ be the space of functions $\varphi\in C(\R, \C)$ such that  
\begin{enumerate}
\item[i)] the integral defining $\scrL\varphi$ converges absolutely if $\Re(s) \ge 2 \pi N$ for some $N \in \N$,  
\item[ii)]  the function $\scrL\varphi$ has an analytic continuation to $\{s \in \HH, \Re(s) >-2 \pi n_0-\epsilon\}$ and can be continuously extended to $\{s \ne 0; s\ge -2 \pi n_0 \}$
\item[iii)] the following series converges:  
\begin{equation}\label{Ff}
\sum_{\substack{n \ge N \\ n \neq 0}} |a(n)| (\scrL|\varphi|)\left (2 \pi n\right ).
\end{equation}
\end{enumerate}
We are now able to define the  $L$-series and recall some results from \cite{DLRR}.
\begin{definition}\label{def:Lf}
Let $f$ be a function on $\HH$ given by 
the Fourier expansion \eqref{FourEx}. 
The $L$-series of $f$ is defined to be the map
$\Lambda_f\colon \mathcal G_f \to \C$ such that, for $\varphi\in \scrG_{f}$, 
\begin{equation}\label{Lseries}
\Lambda_f(\ph)=\sum_{\substack{n \ge N \\ n \neq 0}} a_f(n) (\scrL \ph)(2 \pi n). 
\end{equation}
\end{definition}
Furthermore, for $\re(z)>0$, we recall the generalised exponential integral by
\begin{equation}\label{E-G}
E_p(z):=z^{p-1}\Gamma(1-p,z)=\int_{1}^{\infty}\frac{e^{-zt}}{t^{p}}\mathrm{d}t
\end{equation} The function $E_p(z)$ has an analytic continuation to $\mathbb C \setminus (-\infty, 0]$ as a function of $z$ to give the \emph{principal branch} of $E_p(z)$. 
Specifically, from now on we will always consider the principal branch of the logarithm, so that $-\pi< \arg(z) \le \pi$. Then, we define the analytic continuation of $E_p(z)$  as in (8.19.8) and (8.19.10) of \cite{NIST} to be: 
\begin{equation}\label{EpAnCont}
E_p(z)= \begin{cases} z^{p-1}\Gamma(1-p)-\sum\limits_{0\leq k}\frac{(-z)^k}{k!(1-p+k)} \qquad &\text{for $p \in \mathbb C -\mathbb N$},\\ 
\frac{(-z)^{p-1}}{(p-1)!}(\psi(p)-\log(z))-\sum\limits_{0\leq k \neq p-1}\frac{(-z)^k}{k!(1-p+k)} \qquad &\text{for $p \in \mathbb N$}.
\end{cases}
\end{equation} 
 Since the two series on the right hand side of \eqref{EpAnCont} give entire functions, we can continuously extend $E_p(z)$ to $\mathbb R_{<0}$.  
By (8.11.2) of \cite{NIST}, we also have the bound \begin{equation}\label{boundE}E_p(z)=O(e^{-z}), \qquad \text{
as $z \to \infty$  in the wedge $\arg(z)<3 \pi/2.$}
\end{equation}
A lemma that will be crucial is the sequel is:
\begin{lemma}[\cite{DLRR}]\label{bend}
 If $\im(w)>0$, then we have
\begin{equation}\label{bendeq}
i^{a} E_{1-a}(w)=
\int_i^{i+\infty} e^{iwz} z^{a-1} dz.
\end{equation}
for all $a \in\mathbb R.$  If $\im(w)=0$ and $\re(w)>0$, then \eqref{bendeq} holds for all $a<0.$ 
\end{lemma}

Let $S_k^!(N)$ denote the space of weakly holomorphic cusp forms of weight $k$ for $\Gamma_0(N)$. 
Suppose that $f\in S_k^!(N)$ has Fourier expansion \eqref{FourEx} with respect to the cusp at $\infty$.
By \cite[Lemma~3.4]{BF}, there exists a constant $C_f>0$ such that
\begin{equation} \label{coeffbound}
a_f(n)=O \left (e^{C_f \sqrt{n}}\right ),
\qquad \text{as $n \to \infty$}.
\end{equation}
The $L$-series of $f$ is then defined to be the map
$\Lambda_f\colon \mathcal G_f  \to \C$ given in Definition \ref{def:Lf}.

To describe the $L$-values and derivatives which we are interested in we consider the 
family of test functions given by \eqref{varphInt} and then set
\begin{equation}\label{L(f}
\Lambda(f, s):=\Lambda_f(\varphi_s^0)= \sum_{\substack{n=-n_0 \\  n \ne 0}}^\infty a_f(n) E_{1-s}\left (\frac{2 \pi n}{\sqrt N} \right ).
\end{equation}

\begin{remark}\label{justifi} 
Though more similar in appearance to the usual $L$-series than \eqref{Lseries}, we do not consider $\Lambda(f, s)$ as the ``canonical'' $L$-series of $f$, because, in contrast to $\Lambda_f(\varphi)$ (see Th. 3.5 of \cite{DR}),  it does not satisfy a functional equation with respect to $s$. We formulate our results in terms of $\Lambda(f, s)$ to incorporate it into the setting of \cite{BFI}
and Zagier's formula mentioned in the introduction. The choice of $\Lambda$, rather than $L$ in the notation hints at the analogy with the ``completed" version of the classical $L$-series, rather than with the $L$-series itself. 
\end{remark}

By the proof of Lemma 4.1 of \cite{DR}, or directly, we see that, for $\re(w)>-\epsilon$, $\varphi_s^w \in \mathcal G_f$ and 
\begin{equation}\label{completedwh}
\Lambda_f(\ph_s^{w})=
N^{\frac{s}{2}}\sum_{\substack{n\ge -n_0 \\ n \not =0}}a_f(n)\int_{\frac{1}{\sqrt{N}}}^{\infty}e^{-2 \pi n t-wt }t^{s-1}dt
=\sum_{\substack{n \ge -n_0 \\ n \not =0}} a_f(n)E_{1-s}\left ( \frac{2 \pi n +w}{\sqrt{N}} \right ).
\end{equation}
Because of \eqref{boundE} and the trivial bound for $a_f(n)$, the series $\sum_{\substack{n>0}} a_f(n)E_{1-s}( (2 \pi n +w)/\sqrt{N})$ converges absolutely and uniformly in compact subsets of $\{w \in \HH; \re(w) >-\epsilon\}$, for each fixed $s \in \C$. Since, in addition, $E_{1-s}(z)$ is  continuous from above at each $z \in \mathbb R_{<0}$, we deduce, by comparing with \eqref{L(f}, that
$$\lim_{x \to 0^+}\Lambda_f(\ph_s^{ix})=\Lambda(f, s).$$ 
Let now $s \in \R$ and $x>0$. By Lemma \ref{bend}, followed by a change of variables and \eqref{FourEx}, the sum \eqref{completedwh} becomes
\begin{equation}\label{preform} i^{-s}\sum_{\substack{n \ge -n_0 \\ n \not =0}}a_f(n)\int_{i}^{i+\infty}e^{\frac{(2 \pi n +ix)iz}{\sqrt{N}} }z^{s-1}dz\\
=
i^{-s}N^{s/2}\int_{\frac{i}{\sqrt{N}}}^{\frac{i}{\sqrt{N}}+\infty}e^{-xz}f(z)z^{s-1}dz.
\end{equation}
With the periodicity of $f$, we see that the last integral equals
$$
\sum_{n=0}^{\infty}\int_{\frac{i}{\sqrt{N}}+n}^{\frac{i}{\sqrt{N}}+n+1}e^{-xz}f(z)z^{s-1}dz=
\int_{\frac{i}{\sqrt{N}}}^{\frac{i}{\sqrt{N}}+1}e^{-xz}f(z)\zeta \left (1-s, \frac{ix}{2\pi}, z \right )dz,$$
where
$$\zeta(s, a, z):=\sum_{m=0}^{\infty} e^{2 \pi i m a}(z+m)^{-s}$$
is the Lerch zeta function, which is well-defined since $x>0$. Therefore, we have the following:
\begin{proposition}\label{completedfinwh} For each $f \in S^!_k(N)$ 
and for each $x>0$ and $s \in \mathbb R$, we have
$$\Lambda_f(\ph_s^{ix})= i^{-s}N^{\frac{s}{2}} \int_{\frac{i}{\sqrt{N}}}^{\frac{i}{\sqrt{N}}+1}e^{-xz}f(z) \zeta \left (1-s, \frac{ix}{2\pi}, z \right ) dz.$$
\end{proposition}

\section{Derivatives of $\Lambda(f, s)$\label{sec:derivatives}}
Let $m$ be a positive integer. By $\Lambda^{(m)}_f(\ph_s^{w})$ we denote the $m$th derivative with respect to $s$.  Equation \eqref{completedwh} implies that 
\begin{equation}\label{derL}
\Lambda^{(m)}_f(\ph_s^{w})|_{s=\frac{k}{2}}
=\sum_{\substack{n \ge -n_0 \\ n \not =0}}a_f(n)\frac{d^m}{ds^m}\left (E_{1-s}\left ( \frac{2 \pi n +w}{\sqrt{N}} \right )\right)\Big |_{s=\frac{k}{2}}.
\end{equation}
By the absolute and uniform, in $w$ with $\re(w)>-\epsilon$, convergence of the piece of this series with $n>0$, we deduce that the limit as $w \to 0$ (from above) exists and, with \eqref{L(f}, we have
\begin{equation}\label{derivwh} \lim_{x \to 0^+} \left ( \Lambda^{(m)}_f(\ph_s^{ix})|_{s=\frac{k}{2}}\right )=
\sum_{\substack{n \ge -n_0 \\ n \not =0}}a_f(n)\frac{d^m}{ds^m}\left (E_{1-s}\left ( \frac{2 \pi n}{\sqrt{N}} \right )\right)\Big |_{s=\frac{k}{2}}=
\Lambda^{(m)}(f, k/2).
\end{equation}
On the other hand, we have
\begin{multline*}\frac{d^m}{ds^m}\left ( (i/\sqrt{N})^{-s}\zeta \left (1-s, \frac{ix}{2\pi}, z \right ) \right ) \Big |_{s=\frac{k}{2}}\\=(-1)^m\left (\frac{\sqrt{N}}{i} \right)^{\frac{k}{2}}\sum_{j=0}^m (-1)^j \binom{m}{j}\log \left (\frac{\sqrt{N}}{i} \right )^j \zeta^{(m-j)}\left (1-\frac{k}{2}, \frac{ix}{2 \pi}, z \right ).
\end{multline*}
Using \eqref{derivwh} and Prop. \ref{completedfinwh}, we deduce that
\begin{multline}\label{derivfinwf} 
\Lambda^{(m)}(f, k/2)=(-1)^{m} \left (\frac{\sqrt{N}}{i}\right )^{\frac{k}{2}}\sum_{j=0}^m \binom{m}{j}\log \left (\frac{i}{\sqrt{N}}\right )^j \times \\
\lim_{x \to 0^+}
\int_{\frac{i}{\sqrt{N}}}^{\frac{i}{\sqrt{N}}+1}e^{-xz}f(z) \zeta^{(m-j)} \left (1-\frac{k}{2}, \frac{ix}{2\pi}, z \right ) dz.
\end{multline}
We now use  (8) of Sect. 1.11 of \cite{Er} according to which, for $z \in \mathbb H$, $s \not \in \mathbb N$ and $x>0$ small enough, we have
\begin{equation}\label{Erd} e^{-xz}\zeta \left (s, \frac{ix}{2 \pi}, z \right )=\Gamma(1-s) x^{s-1}+ \sum_{r=0}^{\infty} \zeta(s-r, z)\frac{(-x)^r}{r!},\end{equation}
where $\zeta(s, w)$ is the Hurwitz zeta function. This gives, for every $\ell \in \mathbb N,$
\begin{equation}\label{Erd'} e^{-xz}\zeta^{(\ell)} \left (s, \frac{ix}{2 \pi}, z \right )=\sum_{j=0}^{\ell}(-1)^j 
\Gamma^{(j)}(1-s) x^{s-1} \log^{j}x + \sum_{r=0}^{\infty} \zeta^{(\ell)}(s-r, z) \frac{(-x)^r}{r!}\end{equation}
and thus,
\begin{equation*}\label{Erd0} e^{-xz}\zeta^{(\ell)} \left (1-\frac{k}{2}, \frac{ix}{2 \pi}, z \right )=\sum_{j=0}^{\ell}(-1)^j x^{-k/2}
\Gamma^{(j)}(\frac{k}{2}) \log^{j}x + \sum_{r=0}^{\infty} \zeta^{(\ell)}(1-\frac{k}{2}-r, z) \frac{(-x)^r}{r!}.\end{equation*}
This implies that, for each $j \in \N$, we have
\begin{multline*}
\int_{\frac{i}{\sqrt{N}}}^{\frac{i}{\sqrt{N}}+1}e^{-xz}f(z) \zeta^{(\ell)} \left (1-\frac{k}{2}, \frac{ix}{2\pi}, z \right ) dz =
\left ( \sum_{j=0}^{\ell}(-1)^j\Gamma^{(j)}\left (\frac{k}{2}\right ) x^{-\frac{k}{2}} \log^j x \right )
\int_{\frac{i}{\sqrt{N}}}^{\frac{i}{\sqrt{N}}+1}f(z)dz\\
+\sum_{r=0}^{\infty}\frac{(-x)^r}{r!}\int_{\frac{i}{\sqrt{N}}}^{\frac{i}{\sqrt{N}}+1}f(z)\zeta^{(\ell)}\left ( 1-\frac{k}{2}-r, z\right )dz.
\end{multline*}
Since $f$ has a zero constant term in its Fourier expansion, it follows that 
\begin{equation}\label{transl}
\int_{i/\sqrt{N}}^{i/\sqrt{N}+1}f(z)dz=0.
\end{equation}
Therefore
\begin{equation}
\lim_{x \to 0^+} 
\int_{\frac{i}{\sqrt{N}}}^{\frac{i}{\sqrt{N}}+1}e^{-xz}f(z) \zeta^{(\ell)} \left (1-\frac{k}{2}, \frac{ix}{2\pi}, z \right ) dz =\int_{\frac{i}{\sqrt{N}}}^{\frac{i}{\sqrt{N}}+1}f(z)\zeta^{(\ell)}
\left ( 1-\frac{k}{2}, z\right )dz.
\end{equation}
This, combined with \eqref{derivfin}, proves Theorem \ref{mainwf}.
In the case of weight $2$ it simplifies to
\begin{corollary}\label{cormainwf} For each $f \in S^!_2(N)$ such that $f|_2W_N=f$,
we have
\begin{equation*}
\Lambda'(f, 1)=\sqrt{N} i\int_{\frac{i}{\sqrt{N}}}^{\frac{i}{\sqrt{N}}+1} f(z)\left ( \log(\Gamma(z))+(\log (\sqrt{N})-\pi i/2)z \right ) dz.
\end{equation*}
\end{corollary}
\begin{proof} If $k=2$ and $m=1 $, the formula of the theorem becomes
\begin{equation}\label{wt2}
\Lambda'(f, 1)=\sqrt N i \left (
\log(i/\sqrt{N})\int_{\frac{i}{\sqrt{N}}}^{\frac{i}{\sqrt{N}}+1}f(z) \zeta (0, z) dz+
\int_{\frac{i}{\sqrt{N}}}^{\frac{i}{\sqrt{N}}+1}f(z) \zeta' (0, z) dz
\right ).\end{equation}
The well-known identity $\zeta(0, z) = 1/2 - z$ and \eqref{transl} imply that the first integral equals
$$-\int_{\frac{i}{\sqrt{N}}}^{\frac{i}{\sqrt{N}}+1} f(z)zdz.$$
For the second integral, we combine \eqref{transl} with the identity (see, e.g. (10) of 1.10 of \cite{Er})
$$\zeta'(0, z)=\log(\Gamma(z))-\frac{1}{2}\log (2 \pi).$$
From those formulas for the two integrals we deduce the corollary.
\end{proof}
Finally, we comment on the relation between Th.~\ref{cormain} (applying to holomorphic cusp forms) and Cor.~\ref{cormainwf} (applying to weakly holomorphic ones). Since a holomorphic cusp form is, of course, weakly holomorphic, Cor.~\ref{cormainwf} applies to it too and one might expect the two formulas to agree completely. However, 
the subject of Th.~\ref{cormain} is a different $L$-series from the $\Lambda(f, s)$ appearing in Cor.~\ref{cormainwf}, namely $L_f^*(s)$. They both originate in the more general $\Lambda_f(\varphi)$ but they are not quite the same, $L_f^*(s)$  being simply a ``symmetrised'' version of $\Lambda(f, s).$ This explains why the formulas are identical except for the factor of $2$ in the formula for the central derivative of $L_f^*(s)$.

\section{L-functions associated with cusp forms and their derivatives}\label{class}
The case of classical cusp forms and their L-functions can be accounted for by the same approach. However the setting must be slightly adjusted, ultimately because of the lack of a functional equation for $\Lambda(f, s)$ when $f$ is weakly holomorphic, as discussed in Remark \ref{justifi}.  

Specifically, we let $f$ be a holomorphic cusp form of weight $k$ for $\Gamma_0(N)$ with a Fourier expansion  \begin{equation}\label{FourEx0}
f(z) = \sum_{\substack{n > 0}} a_f(n) e^{2\pi inz},
\end{equation} and such that $$f|_kW_N=f, \qquad  
\text{for $W_N=\sm 0 & -1/\sqrt{N} \\\sqrt{N} & 0 \esm $.}$$
We recall the classical integral expression for the completed $L$-function of $f$:
\begin{equation}\label{completed0}
\begin{aligned}
L^*_f(s):=\left (\frac{\sqrt{N}}{2 \pi}\right )^s \Gamma(s) L_f(s)&=
N^{\frac{s}{2}} \int_{1/\sqrt{N}}^{\infty}f(it)t^{s-1}dt+i^k
N^{\frac{k-s}{2}} \int_{1/\sqrt{N}}^{\infty}f(it)t^{k-1-s}dt \\
&=
\sum_{n>0}a_f(n)E_{1-s}(2 \pi n/\sqrt{N})+i^k
\sum_{n>0}a_f(n)E_{s-k+1}(2 \pi n/\sqrt{N})
\end{aligned}
\end{equation}
We observe that, thanks to \eqref{boundE}, this converges for all $s \in \C.$ The 
completed $L$-function can be recast in terms of the $L$-series formalism of \cite{DR} and the family of test functions given in \eqref{varphInt}. Indeed, if $\re(w) >-\epsilon$, we have,
\begin{multline}\label{completed}
L_f(\ph_s^{w}+i^k\ph_{k-s}^{w})=
N^{\frac{s}{2}}\sum_{n>0}a_f(n)\int_{\frac{1}{\sqrt{N}}}^{\infty}e^{-2 \pi n t-wt }t^{s-1}dt+i^k
N^{\frac{k-s}{2}}\sum_{n>0}a_f(n)\int_{\frac{1}{\sqrt{N}}}^{\infty}e^{-2 \pi n t-wt }t^{k-1-s}dt  \\
=\sum_{n>0}a_f(n)\int_{1}^{\infty}e^{-\frac{(2 \pi n +w)t}{\sqrt{N}} }t^{s-1}dt+i^k
\sum_{n>0}a_f(n)\int_{1}^{\infty}e^{-\frac{(2 \pi n +w)t}{\sqrt{N}}}t^{k-1-s}dt
\end{multline}
As in the previous section (but more easily, since we do not have any terms with $n<0$), the series converges absolutely and uniformly in compact subsets of $\{w \in \HH; \re(w) >-\epsilon\}$, for each fixed $s \in \C$. Hence, comparing with \eqref{completed0}, we see that
$$\lim_{x \to 0^+}L_f(\ph_s^{ix}+i^k\ph_{k-s}^{ix})=L^*_f(s).$$ 
Let now $s \in \R$ and $w \in \HH$ with $\re(w)>-\epsilon$. By Lemma \ref{bend}, followed by a change of variables and \eqref{FourEx0}, the sum \eqref{completed} becomes
\begin{multline}i^{-s}\sum_{n>0}a_f(n)\int_{i}^{i+\infty}e^{\frac{(2 \pi n +w)iz}{\sqrt{N}} }z^{s-1}dz+i^k
i^{s-k}\sum_{n>0}a_f(n)\int_{i}^{i+\infty}e^{\frac{(2 \pi n +w)t}{\sqrt{N}}}t^{k-1-s}dz\\
=
i^{-s}N^{s/2}\int_{i/\sqrt{N}}^{i/\sqrt{N}+\infty}e^{iwz}f(z)z^{s-1}dz+
i^{s}N^{(k-s)/2}\int_{i/\sqrt{N}}^{i/\sqrt{N}+\infty}e^{iwz}f(z)z^{k-1-s}dz.
\end{multline}
This is a 'symmetrised' analogue of \eqref{preform} and therefore, working similarly to the last section, we can deduce the following analogue of Prop. \ref{completedfinwh}:
\begin{proposition}\label{completedfin} Let $f \in S_k(N)$ such that $f|_kW_N=f$. For each $w \in \mathbb H$ with $\re(w)>-\epsilon$ and each $s \in \mathbb R$, we have
$$L_f(\ph_s^{w}+i^k\ph_{2-s}^{w})=\int_{\frac{i}{\sqrt{N}}}^{\frac{i}{\sqrt{N}}+1}e^{iwz}f(z) \left ( i^{-s}N^{\frac{s}{2}}\zeta \left (1-s, \frac{w}{2\pi}, z \right ) +
i^{s}N^{\frac{k-s}{2}}\zeta \left (s-k+1, \frac{w}{2\pi}, z \right )\right ) dz.$$
\end{proposition}

To pass to derivatives, we let $m$ be a positive integer. Equation \eqref{completed} implies that 
$$
L^{(m)}_f(\ph_s^{w}+i^k\ph_{2-s}^{w})|_{s=\frac{k}{2}}
=(1+i^{2m+k})\sum_{n>0}a_f(n)\int_{1}^{\infty}e^{-\frac{(2 \pi n +w)t}{\sqrt{N}} }t^{\frac{k}{2}-1} \log^m t dt.
$$
which is the analogue of \eqref{derL} and thus, we can work in an entirely analogous way to the last section to obtain
\begin{multline}\label{derivfin} 
(L^*_f)^{(m)}\left (\frac{k}{2}\right )=(i^{2m}+i^k) \left (\frac{\sqrt{N}}{i}\right )^{\frac{k}{2}}\sum_{j=0}^m \binom{m}{j}\log \left (\frac{i}{\sqrt{N}}\right )^j \times \\
\lim_{x \to 0^+}
\int_{\frac{i}{\sqrt{N}}}^{\frac{i}{\sqrt{N}}+1}e^{-xz}f(z) \zeta^{(m-j)} \left (1-\frac{k}{2}, \frac{ix}{2\pi}, z \right ) dz.
\end{multline}
Applying (8) of Sect. 1.11 of \cite{Er} as in the last section implies that this equals
\begin{equation*}
(i^k+i^{2m})  \left (\frac{\sqrt{N}}{i}\right )^{\frac{k}{2}} \sum_{j=0}^m  \binom{m}{j}\log^j \left (\frac{i}{\sqrt{N}}\right ) 
\int_{\frac{i}{\sqrt{N}}}^{\frac{i}{\sqrt{N}}+1}f(z) \zeta^{(m-j)} \left (1-\frac{k}{2}, z \right ) dz.
\end{equation*}
Since $L^*_f(s)=(\sqrt{N}/(2 \pi))^s\Gamma(s)L_f(s),$ this gives:
\begin{theorem}\label{main} Let $m$ be a positive integer. For each $f \in S_k(N)$ such that $f|_kW_N=f$ and $L^{(j)}_f(k/2)=0$ for $j<m$
we have
\begin{equation*}
L^{(m)}_f\left ( \frac{k}{2}\right)=\frac{i^k+i^{2m}}{\left ( \frac{k}{2} - 1 \right)!} (-2 \pi i )^{\frac{k}{2}}\sum_{j=0}^m  \binom{m}{j}\log^j \left (\frac{i}{\sqrt{N}}\right ) 
\int_{\frac{i}{\sqrt{N}}}^{\frac{i}{\sqrt{N}}+1}f(z) \zeta^{(m-j)} \left (1-\frac{k}{2}, z \right ) dz.
\end{equation*}
\end{theorem}
Theorem \ref{cormain} follows from this exactly as in Cor. \ref{cormainwf} once we take into account that, if $k=2$ and $f|_2W_N=f$, we automatically have $L_f(1)=0$ by the classical functional equation for $f \in S_2(N).$

\section{Computational and algorithmic aspects\label{sec:comp}}

Consider first the special case of a holomorphic cusp form $f$ of
weight $k=2$ and level $N$, which is invariant under the Fricke
involution $W_{N}$. Suppose that $f$ has a Fourier expansion of the form \eqref{FourEx0}.
It is clear from \eqref{completed0} and symmetry
that the central value $L_{f}^{*}(1)$ is zero and the $r$th central derivative is is zero, if $r$
is even, and 
\[
(L_{f}^*)^{(r)}(1)=2r!\sum_{n>0}a_{f}(n)E_{0}^{r}\left(\frac{2\pi n}{\sqrt{N}}\right),
\]
if $r$ is odd. Here 
\[
E_{s}^{r}(z)=\frac{1}{r!}\int_{1}^{\infty}e^{-zt}(\log t)^{r}t^{-s}dt
\]
is $(-1)^{r}/r!$ times the $r$-th derivative of $E_{s}(z)$ with
respect to $s$. It is initially defined for
 $\Re(z)>0$ and can be extended to $\HH \cup \R_{<0}$ via \eqref{eq:Em_mn_formula} and
  \eqref{eq:E1m} below. Using integration by parts it can be shown that $E_{0}^{r}(z)=\frac{1}{z}E_{1}^{r-1}(z)$,
which leads to the expression 
\begin{equation}
(L_{f}^*)^{(r)}(1)=\frac{\sqrt{N}}{\pi}r!\sum_{n>0}a(n)\frac{1}{n}E_{1}^{r-1}\left(\frac{2\pi n}{\sqrt{N}}\right).\label{eq:Lp_as_E1_sum}
\end{equation}
This expression was first obtained by Buhler, Gross and Zagier in
\cite{MR777279}, where the authors used the following expression
to evaluate $E_{1}^{m}(z)$ for any $m\ge1$ and $z>0$
\begin{equation}
E_{1}^{m}(z)=G_{m+1}=P_{m+1}(-\log z)+\sum_{n\ge1}\frac{(-1)^{n-m-1}}{n^{m+1}n!}z^{n}.\label{eq:E1m}
\end{equation}
Here $P_{r}(x)$ is a polynomial of degree $r$ and if we write $\Gamma(1+z)=\sum_{n\ge0}\gamma_{n}z^{n}$
then 
\[
P_{r}(t)=\sum_{j=0}^{r}\gamma_{r-j}\frac{t^{j}}{j!}.
\]
Extending this method to weights $k\ge4$ and weakly holomorphic modular
forms is immediate. If $f\in S_{k}^{!}(N)$ has Fourier expansion
at infinity of the form \eqref{FourEx}
then the analogue of \eqref{completed0} is \eqref{L(f}.
Upon differentiating \eqref{L(f} $r$ times with respect to $s$ and setting
$s=k/2$ leads to 
\begin{equation}
\Lambda^{(r)}(f,k/2)=r!\sum_{\substack{n \ge -n_0 \\ n \neq 0}}a_{f}(n)E^{r}_{1-k/2}\left(\frac{2\pi n}{\sqrt{N}}\right),\label{eq:Lderweak}
\end{equation}
where we note that for a holomorphic $f$ we have $(L_{f}^*)^{(m)}(k/2)=(1+i^{k+2m})\Lambda^{(m)}(f,k/2)$.
It follows that we need to evaluate $E_{-n}^{r}$ where $n=k/2-1.$
To compare the complexity of these computations with the weight $2$
case we note that Milgram \cite[(2.22)]{MR777276} showed that 
\begin{equation}
E_{-n}^{m}(z)=\frac{\Gamma(n+1)}{z^{n+1}}\left[e^{-z}\sum_{l=0}^{n-m}\frac{z^{l}}{l!}\xi_{l,n}^{m}+\sum_{l=1}^{m}\xi_{0,n}^{l-1}E_{1}^{m-l}(z)\right],\label{eq:Em_mn_formula}
\end{equation}
where where $\xi_{l,n}^{j}$ are constants independent of $z$ and
can be precomputed. Using this together with (\ref{eq:E1m}) it follows
that the computation essentially reduces to that of a finite sum of
polynomials and an infinite rapidly convergent sum. 

It is also worth to mention here that the general algorithm to compute
values and derivatives of Motivic $L$-functions introduced by Dokchitser
in \cite{MR2068888} and implemented in PARI/GP \cite{PARI2}, essentially
reduces to that described above in the case of holomorphic modular
forms. Furthermore, in both \cite{MR777279} and \cite{MR2068888}
the authors make additional use of asymptotic expansions to speed
up computations of $E_{-n}^{m}(z)$ for large $z$. 

\subsection{The new integral formula}

Let $f\in S_{k}^{!}(\Gamma_{0}(N))$ be a weakly holomorphic cusp
form of even integral weight $k$ and that satisfies $f|_{k}W_{N}=f$.
Then Theorem 1.1 implies that 

\[
\Lambda^{(m)}(f,k/2)=i^{2m-k/2}N^{k/4}\sum_{j=0}^{m}{m \choose j}
\log^{j}\left(\frac{i}{\sqrt{N}}\right)\int_{i/\sqrt{N}}^{i/\sqrt{N}+1}f(z)\zeta^{(m-j)}\left(1-\frac{k}{2},z\right)dz,
\]
\[
\]
where $\Lambda(f,s)$ is defined in (\ref{L(fInt)}).
When computing these values it is clear that the main CPU time is
spent on computing integrals of the form 
\[
I_{r}(f)=\int_{0}^{1}f(x+i/\sqrt{N})\zeta^{(r)}\left(1-k/2,x+i/\sqrt{N}\right)dx,\ 0\le r\le m.
\]
The cusp form $f$ is given in terms of the Fourier expansion \eqref{FourEx}
for some $n_{0}\ge0$. To evaluate $f(x+i/\sqrt{N})$ up to a precision
of $\varepsilon=10^{-D}$ for all $x\in[0,1]$ we can truncate the
Fourier series at some integer $M>0$. The precise choice of $M$
depends on the available coefficient bounds. In case $f$ is holomorphic
then Deligne's bound can be used to show that we can choose $M$ such
that 
\[
M > c_{1}k\sqrt{N}\log M +\sqrt{N}(c_{2}D+c_{3}\log(\sqrt{N}(k/2)!))+c_{4}
\]
for some explicit positive constants $c_{1},c_{2},c_{3}$ and $c_{4}$,
independent of $N,D$ and $k$. However, if $f$ is not holomorphic
then we only have the non-explicit bound \eqref{coeffbound} and $M$
must satisfy 
\[
M > c_{1}'\sqrt{N}\sqrt{M} + c_{2}'\sqrt{N}D+c_{3}'\sqrt{N}\log N,
\]
where $c_{1}',$ $c_{2}',$ $c_{3}'$ and $c_{4}'$ are positive constants
that depend on $f$ and can be computed in special cases using Poincaré
series. In both cases we From both inequalities above it is clear 
that as the level or
weight increases we need a larger number of coefficients, which increases
the number of arithmetic operations needed. Note that the working
precision might also need to be increased due to cancellation errors.
To evaluate the Hurwitz zeta function and its derivatives it is possible
to use, for instance, the Euler -- Maclaurin formula 
\begin{equation*}
\zeta(s,z)=\sum_{n=0}^{M-1}\frac{1}{(n+z)^{s}}+\frac{(z+M)^{1-s}}{s-1}+\frac{1}{(z+M)^{s}}\left(\frac{1}{2}+\sum_{l=1}^{L}\frac{B_{2l}}{(2l!)}\frac{(s)_{2l-1}}{(z+M)^{2l-1}}\right)+\text{Err}(M,L)
\end{equation*}
where $M,L\ge1$ and where the error term $\text{Err}(M,L)$ can be
explicitly bounded. For more details, including proof and analysis
of rigorous error bounds and choice of parameters 
see \cite{MR3350381}, where the generalisation
to derivatives $\zeta^{(r)}(s,z)$ is also included. In our case $s=1-k/2$
and $z=x+i/\sqrt{N}$ with $0\le x\le1$. It is easy to use Theorem
1 of \cite{MR3350381} to show that if $M>1$ and $L>k/4$ then 
\[
\left|\text{Err}(M,L)\right|\le\frac{2M^{2k}}{(2\pi M)^{2L}}\frac{|(1-k/2)_{2L}|}{L-k/4},
\]
where $(s)_{m}=s(s+1)\cdots(s+m-1)$ is the usual Pochammer symbol.
Furthermore, if the right-hand side above is denoted by $B$ then
it can be shown that the error in the Euler -- Maclaurin formula
for the $r$th derivative can be bounded by $B\cdot r!\log(8(M+1))^{r}$.
In \cite{MR3350381} it is observed that to obtain $D$ digits of
precision we should choose $M\sim L\sim D,$ meaning that the number
of terms in both sums are proportional to $D$. It is also clear that
as $k$ or $r$ increases we will need larger values of $M$ and $L$. 
\begin{example}
Consider $f\in S_{2}(37)$ and standard double precision, i.e. $53$
bits or $15$ (decimal) digits. Then a single evaluation of $f(x+i/\sqrt{37})$
takes $271\mu s$ while $\zeta^{(r)}(0,x+i/\sqrt{37})$
takes $2\mu s$, $114\mu s$, $124\mu s$,
$171\mu s$ for $r=1,2,3$ and $20$, respectively. 
\end{example}

\subsection{Comments on the implementation}

There are a few simple optimisations that can be applied immediately
to decrease the number of necessary function evaluations. 
\begin{itemize}
\item Replace the sum of integrals by $\int_{0}^{1}f(x+i/\sqrt{N})Z_{m}(x+i/\sqrt{N})dx$,
where 
\[
Z_{m}(z)=\sum_{j=0}^{m}{m \choose j}\log^{j}\left(\frac{i}{\sqrt{N}}\right)\zeta^{(m-j)}\left(1-\frac{k}{2},z\right).
\]
\item If $f(z)$ has real Fourier coefficients then $f(1-x+i/\sqrt{N})=\overline{f(x+i/\sqrt{N})}$,
which is very useful as we can choose the numerical integration
method with nodes that are symmetric with respect to $x=1/2$.
\item If we need to compute $\Lambda^{(r)}(f,k/2)$ for a sequence of $r$s,
then function values of $f$ and lower derivatives $\zeta^{(j)}$
can be cached in each step provided that the we use the same nodes
for the numerical integration. 
\end{itemize}
As the main goal of this paper is to present a new formula and not to present
an optimised efficient algorithm as such, we have implemented all algorithms
in SageMath using the mpmath Python library for the Hurwitz zeta function
evaluations as well as for the numerical integration using Gauss - Legendre quadrature. 
The implementation used to calculate the examples below can be found in a Jupyter notebook
which is available from \cite{gitHubEx}.

\subsection{Examples of holomorphic forms}

To demonstrate the veracity of the formulas in this paper we first
present a comparison of results and indicative timings between the
new formula in this paper and Dokchiter's algorithm in PARI (interfaced
through SageMath). 

Table \ref{tab:lderholo} includes three holomorphic cusp forms $\text{37.2.a.a}$,
$\text{127.4.a.a}$ and $\text{5077.2.a.a}$, labelled according to
the LMFDB \cite{lmfdb}. These are all invariant under the Fricke
involution and it is known that the analytic ranks are $1,$ $2$
and $3$, respectively. The last column gives the difference between the
values computed by Dokchitser's algorithm and the integral formula. 

\begin{table}[!htbp]
\centering
\footnotesize
\begin{tabular}{|r|l|r|l|l|r|l|r|r|}
\toprule 
$N$ & $k$ & \multicolumn{1}{c}{Label} & $r$ & Dokchitser/PARI & Time (ms) & Integral formula & Time (ms) & Error\tabularnewline
\midrule
$37$ & $2$ & 37.2.a.a & $1$ & 0.296238908699801 & 18 & 0.2962389086998011 & 49 &$6\cdot10^{-17}$\tabularnewline
127 & $4$ & 127.4.a.a & $2$ & 7.83323138624802 & 42 & 7.8332313863855996+ & 186 & $1\cdot10^{-10}$\tabularnewline
5077 & $4$ & 5077.2.a.a & $3$ & 117.837959237940 & 212 & 117.83795923792273+ & 2000 & $2\cdot10^{-11}$\tabularnewline
\bottomrule
\end{tabular}
\caption{Central derivatives $(L_{f}^*)^{(r)}(k/2)$ for $f\in S_k(\Gamma_0(N))$ \label{tab:lderholo}}
\end{table}
As the level increases we find that $f(x+i/\sqrt{N})$ oscillates
more and more and it is necessary to increase the degree of the Legendre
polynomials used in the Gauss--Legendre quadrature. The comparison
of timings in Table \ref{tab:lderholo} indicates that our new formula is slower than Dokchitser's
algorithm but it is important to keep in mind the latter is implemented
in the PARI C library and is compiled while our formula is simply implemented
directly in SageMath using the mpmath Python library. 
All CPU times presented below are obtained on a 2GHz Intel Xeon Quad Core 
and we stress that the times should not be taken as absolute performance measures
but simply to provide comparisons between different input and parameter values.

\subsection{Examples of weakly holomorphic modular forms}

To construct weakly modular cusp forms we use the Dedekind eta functions 
\[
\eta(\tau)=q^{\frac{1}{24}}\prod_{n\ge1}\left(1-q^{n}\right).
\]
If we define 
\[
\Delta_{2}^{+}(\tau)=(\eta(\tau)\eta(2\tau))^{8}=q-8q^{2}+12q^{3}+64q^{4}+O(q^{5})
\]
and 
\[
j_{2}^{+}(\tau)=(\eta(\tau)/\eta(2\tau))^{24}+24+2^{12}(\eta(2\tau)/\eta(\tau))^{24}=q^{-1}+4372q+96256q^{2}+1240002q^{3}+O(q^{4})
\]
then it can be shown that $\Delta_{2}^{+}\in S_{8}(\Gamma_{0}(2))$
and $j_{2}^{+}\in S_{0}^{!}(\Gamma_{0}(2))$ are both invariant under
the Fricke involution $W_{2}$.
The following holomorphic and weakly holomorphic modular forms of weight $16$ on
$\Gamma_{0}(2)$ were introduced by Choi and Kim \cite{MR3451920}
to study weakly holomorphic Hecke eigenforms. 
\begin{align*}
f_{16,-2}(\tau)&=\Delta_{2}^{+}(\tau)^{2}=q^{2}-16q^{3}+O(q^{4}) \\
 f_{16,-1}(\tau)&=\Delta_{2}^{+}(\tau)^{2}(j_{2}^{+}(\tau)+16)=q+4204q^{3}+O(q^{4})\\
 f_{16,0}(\tau)&=\Delta_{2}^{+}(\tau)^{2}(j_{2}^{+}(\tau)^{2}+16j_{2}^{+}(\tau)-8576)=1+261120q^{3}+O(q^{4})\\
 f_{16,1}(\tau)&=\Delta_{2}^{+}(\tau)^{2}(j_{2}^{+}(\tau)^{3}+16j_{2}^{+}(\tau)^{2}-12948j_{2}^{+}(\tau)-427328)=q^{-1}+7525650q^{3}+O(q^{4})\\
 f_{16,2}(\tau)&=\Delta_{2}^{+}(\tau)^{2}(j_{2}^{+}(\tau)^{4}+16j_{2}^{+}(\tau)^{3}-17320j_{2}^{+}(\tau)^{2}-593536j_{2}^{+}(\tau)-27188524)\\
   &=q^{-2}+140479808q^{3}+O(q^{4})
\end{align*}
and it is easy to see that all of these functions are also invariant under $W_{2}$.
Furthermore, $f_{16,-2},f_{16,-1}\in S_{16}(\Gamma_{0}(2))$ and $f_{16,1},f_{16,2}\in S_{16}^{!}(\Gamma_{0}(2))$
while $f_{16,0}$ is not cuspidal. 

To check the accuracy of our formula in this setting we first consider the holomorphic cusp forms. 
Observe that the unique newform of level 2 and
weight 16 is
\[
f(\tau)=q-128q^{2}+6252q^{3}+16384q^{4}+90510q^{5}+O(q^{6})=f_{16,-1}-128f_{16,-2}.
\]
Using Dokchitser's algorithm we find that $L_{f}^{*}(8)=0.0526855929956408$
while using the integral formula with 53 bits precision we obtain
\begin{align*}
L_{f_{16,-2}}^{*}(8) & =0.00008045589767063483+6\cdot10^{-20}i,\\
L_{f_{16,-1}}^{*}(8) & =0.06298394789748197609+3\cdot10^{-17}i,
\end{align*}
and 
\[
L_{f_{16,-1}}^{*}(8)-128L_{f_{16,-2}}^{*}(8)=0.05268559299564071785+2\cdot10^{-17}i,
\]
which agrees with the value of $L^{*}_f(8)$ above.

Table \ref{tab:lderweakint} gives the values of $\Lambda^{(r)}(f_{16,i},8)$ for the weakly holomorphic modular forms $f_{16,1}$ and $f_{16,2}$, 
computed using the integral formula with $103$ bits working precision. The table contain an indication of timings as well as a
 heuristic error estimate based on a comparison with the same value computed using $203$ bits precision.

To provide some independent verification of the algorithm in the case of weakly modular forms we also implemented the generalisation of the algorithm from 
\cite{MR777279} using 
\eqref{eq:Lderweak} directly with $E_{1-k/2}^{r}$ evaluated using \eqref{eq:Em_mn_formula} and \eqref{eq:E1m}.
The main obstacle with the algorithm modelled on \cite{MR777279} is that the infinite sum in \eqref{eq:E1m} suffers from catastrophic cancellation for large $z$
unless the working precision is temporarily increased within the sum.
The corresponding values of $\Lambda^{(r)}(f_{16,i},8)$ computed using the algorithm with 103 bits starting precision 
are given in Table \ref{tab:lderweaksum} where we also give the corresponding timings as well as an error estimate based on comparison with values 
in Table \ref{tab:lderweakint}.

\begin{table}[!htbp]
\centering
\footnotesize
\begin{tabular}{|r|r|l|r|c|}
\toprule 
$i$ & $r$ & \multicolumn{1}{c|}{$\Lambda^{(r)}(f_{16,i},8)$}
 & T/ms & Err. \tabularnewline
\midrule
1 & 0  & $-0.2035186511755524285671725692737 +1\cdot10^{-31}$ & $204$ & $6\cdot10^{-30}$ \tabularnewline
1 & 1  & $\hphantom{-}1.1597162067012225517004253561026-0.104294509255933530762675132394i$ & $975$ & $9\cdot10^{-30}$ \tabularnewline
1 & 2  & $-0.3329012203856171470128799683152  -0.109371149169408369683239573058i$ & $1790$ & $7\cdot10^{-30}$  \tabularnewline
2 & 0  & $-1.8934024663352144735029014555039 + 1 \cdot10^{-30}$ & $209$ & $1\cdot10^{-27}$ \tabularnewline
2 & 1  & $\hphantom{-}55.394013302380372465449909213930 -0.000407400426780990354541699709i$ & $996$ & $2\cdot10^{-28}$ \tabularnewline
2 & 2  & $-0.1484917546377626240694524994979+0.000137545862921322355701592298i$ & $1880$ & $1\cdot10^{-28}$ \tabularnewline
\bottomrule
\end{tabular}
\caption{$\Lambda^{(r)}(f_{16,i},8)$ computed using the integral formula with 103 bits precision. \label{tab:lderweakint}}
\end{table}

\begin{table}[!htbp]
\centering
\footnotesize
\begin{tabular}{|r|r|l|r|c|}
\toprule 
$i$ & $r$ & \multicolumn{1}{c|}{$\Lambda^{(r)}(f_{16,i},8)$}  & T/ms & Err.\tabularnewline
\midrule
1 & 0 &  $-0.20351865117555238$ & $10$ & $4\cdot 10^{-17}$ \tabularnewline
1 & 1 &  $\hphantom{-}1.15971620670121522423-0.104294509255934i$ & $11\cdot10^{3}$ & $8\cdot 10^{-15}$\tabularnewline
1 & 2  & $-0.33290122038562486306-0.109371149169408i$ &  $21\cdot10^{3}$& $8\cdot 10^{-15}$ \tabularnewline
2 & 0 & $-1.89340246633520092878$ & $11$ & $2\cdot 10^{-14}$ \tabularnewline
2 & 1  & $\hphantom{-}55.3940133023803440437-0.000407400426780990i$ &  $14\cdot10^{3}$& $4\cdot 10^{-14}$\tabularnewline
2 & 2 & $-0.14849175463777442019+0.000137545862921322i$ & $26\cdot10^{3}$ & $2\cdot 10^{-14}$\tabularnewline
\bottomrule
\end{tabular}
\caption{$\Lambda^{(r)}(f_{16,i},8)$ computed using the sum with 103 bits precision. \label{tab:lderweaksum}}
\end{table}

\end{document}